\documentclass{article}
\usepackage[utf8]{inputenc}
\usepackage{yfonts}
\usepackage[sumlimits]{amsmath}
\usepackage{amssymb}
\usepackage{amsthm}
\usepackage{mathrsfs}
\usepackage{txfonts}
\usepackage{faktor, enumerate}
\usepackage{xcolor}
\usepackage{mathtools}
\usepackage{xparse}
\usepackage[top=1.0in, bottom=1in, left=1.4in, right=1.85in]{geometry}
\usepackage{lipsum}
\usepackage{graphicx}
\usepackage{titlesec}
\usepackage{extarrows}

\numberwithin{equation}{section}

\newcommand{\N}{\mathbb{N}}
\newcommand{\Z}{\mathbb{Z}}

\newcommand{\eps}{\varepsilon}

\newcommand{\norm}[1]{\left\lVert#1\right\rVert}

\newcommand{\act}{\curvearrowright}

\newcommand{\defeq}{\vcentcolon=}

\newtheorem{thm}{Theorem}[section]
\newtheorem{defn}[thm]{Definition}
\newtheorem{prop}[thm]{Proposition}

\newtheorem{cor}[thm]{Corollary}
\newtheorem{lem}[thm]{Lemma}
\newtheorem{ex}[thm]{Example}

\renewcommand{\phi}{\varphi}
\renewcommand{\d}{\mathrm{d}}

\date{}

\title{Quantitative Amenability for Actions of Finitely Generated Groups}

\begin{document}
\author{Zihan Xia}
\maketitle

\begin{abstract}
    We generalize the notion of isoperimetric profiles of finitely generated groups to their actions by measuring the boundary of finite subgraphings of the orbit graphing. We prove that like the classical isoperimetric profiles for groups, decay of the isoperimetric profile for an essentially-free action is equivalent to amenability of the action in the sense of Zimmer.  For measure-preserving actions, we relate the isoperimetric profiles of the actions and the group.
\end{abstract}

\section{Introduction}
 Let $G$ be a finitely generated discrete group with a finite symmetric generating set $S$. Recall that the isoperimetric profile of $(G,S)$ is defined as:
\[
\mathcal{I}_{G}(n) = \inf_{\substack{F\subset G\\
|F|\leq n}}
\frac{|\partial F|}{|F|}
\]
where $|\ |$ is the counting measure and $\partial F$ is the boundary set with respect to the generating set $S$. This is closely related to the F\o lner functions, which is the inverse of the isoperimetric profile in some sense. These two can be related by the inequality of Coulhon-Saloff-Coste \cite{C-S-C}. Such isoperimetric profiles for groups have two properties:
\begin{enumerate}
    \item The asymptotics of the isoperimetric profile for a finitely generated group are independent of the chosen generators.
    \item the isoperimetric profile converges to zero if and only if $G$ is amenable.
\end{enumerate}

These isoperimetric profiles and F\o lner functions have been studied by Vershik \cite{folner}, Pittet \cite{pittet2000isoperimetric}, Stankov \cite{BSgroupfolner}, Cavaleri \cite{CAVALERI2018388}, Erschler \cite{Erschler} and many others. In this paper we construct isoperimetric profiles for group actions, and we will prove that they satisfy analogue of the above-mentioned properties (1) and (2).
\\
Now assume that $(X,\mu)$ is a standard probability measure space and $(G,S)$ acting on it. The notion of isoperimetric profile for groups inspires us to construct the similar notion of isoperimetric profile for such action $(G,S)\act (X,\mu)$. A good news is that the boundary of any point stays in its own orbit. Therefore, instead of chosing arbitrarily measurable set $F\subset X$ and computing the boundary ratio, we partitioned each orbit into components of finite points. Since every group action induces an orbit measurable equivalence relation $\mathcal{R}$, we choose the measurable subequivalence relation of $\mathcal{R}$ to be such partition. In this way, the isoperimetric profile for the actions is computed by the average of the boundary ratio of the orbit subequivalence classes, and then take infimum. In the case of measure-preserving action, such formula can be written as the following definition, and we decide to use this formula for the nonmeasure-preserving actions.

\begin{defn}
The isoperimetric profile $\mathcal{I}_{(G,S)\act X}(n)$ of an action $(G,S)\act (X,\mu)$ is defined and as follows: 
\[
\mathcal{I}_{(G,S)\act X}(n)\defeq\inf_{\substack{\theta_n \subset \mathcal{R} \\|\theta_n|\leq n}}
\mu(\partial \theta_n)=\inf_{\substack{\theta_n \subset \mathcal{R} \\|\theta_n|\leq n}}\mu\{x\in X | sx \notin [x]_{\theta_n} \text{ for some } s\in S\}
\]
We will omit $S$ and write $\mathcal{I}_{G\act X}(n)$ if the generating set is fixed and this does not cause confusion.
\end{defn}
In general, we can define the similar isoperimetric profiles for graphings, but currently we don't know if the relationship between the decay of $\mathcal{I}$ and amenability(Theorem \ref{iff}) still holds in this case.

We will prove that if the Radon-Nikodym derivatives of an action are in $L^\infty$, then the asymptotics of the isoperimetric profile of the action are independent of the choice of the generating set up to a multiplicative constant (see Proposition \ref{generatingset}). This is analogous to Property (1) of the isoperimetric profile for groups above. More generally, if the Radon-Nikodym derivatives are in $L^p$ for some $p>1$ then a change of generators will give a polynomial bound. There is also a result which is analogue of property (2), but in terms of Zimmer's amenability for action instead of amenability of groups. So here we give the following main result, that under a mild integrability hypothesis, the decay of the isoperimetric profile of an action characterizes amenability.


\begin{thm}\label{iff}
Let $(G,S)$ be a finitely generated infinite group acting essentially free and ergodically on a standard non-atomic probability measure space $(X,\mu)$, where $\mu$ is non-atomic. Assume the Radon-Nikodym derivative $ds_*\mu/d\mu$ is $L^{1+\delta}$-bounded for each $s\in S$ and some $\delta>0$. Then the action is amenable if and only if $\mathcal{I}_{G\act X}(n)\to 0$.
\end{thm}
This result based on a significant proerty from Connes-Feldman-Weiss \cite{connes_feldman_weiss_1981} that the orbit equivalence relation being hyperfinite is equivalent to the action being amenable in Zimmer's sense. Such property build up a relationship between the isoperimetric profile of an action and the amenability of actions. 
\noindent\textbf{Remark.} There are many examples of actions satisfying the integrability assumpted in Theorem \ref{iff}: Measure preserving actions should be the most common cases. In addition, if $m$ be a probability measure on $S$ such that $\mu$ is $m$-stationary, which then tell us $\mu$ actually has $L^\infty$ Radon-Nikodym derivative with respect to $S$. In particular, this is true for a finitely generated group acting on it's poisson boundary.\\
Next, we study the isoperimetric profiles of probability measure-preserving actions.
\begin{thm}\label{sm}
Let $G$ be a finitely generated group and suppose $G\act(X,\mu)$ is probability measure preserving and essentially free. If $\mathcal{I}_{G\act X}(n)\to 0$ then $G$ is amenable. Further, the converse is true if the group action is ergodic.
\end{thm}
Finally, given an amenable action $G\act (X,\mu)$, it would be interesting to obtain bounds on $\mathcal{I}_{G\act X}$ and its relationship with the classical isoperimetric profile $\mathcal{I}_G$. For probability measure preserving actions, we will prove that isoperimetric profile for the action $G\act X$ is bounded below by the isoperimetric profile of the group $G$. Furthermore, that bound can be reached if there exists a sequence of F\o lner tilings which has isoperimetric ratio asymptotic to the isoperimetric profile of $G$. For the definition of a tiling, see Definition \ref{tiling}. There are many groups which can be tiled by F\o lner sets, and for some groups we can also find F\o lner sets with optimal ratio. For example $\Z^d$ for $d\in\N$ and the Heisenberg group. But we still do not know which group admit F\o lner tilings, never mind optimal ones. However, Downarowicz, Huczek and Zhang \cite{tilings} proved that any amenable group admits a ``multitilings'' whose isoperimetric ratio can be chosen arbitrarily small

\begin{thm}\label{folnertilings}
Let $G$ be a finitely generated infinite amenable group. Then there exists a sequence of multi-tiles $(T_n)_{n\in\N}$ where $T_n=(T_{1n},\dots T_{Nn})$ and $N=N(n)$, such that for each $1\leq i \leq N$, $|T_{in}|\leq n$ and $\max\limits_{1\leq i\leq N}\frac{|\partial T_{in}|}{|T_{in}|}\xrightarrow{n \to \infty} 0$. For every multi-tile $T_n$ and any free probability measure-preserving action $G\act(X,\mu)$, we have
    \begin{equation}
        \mathcal{I}_{G\act X}(n)\leq \max_{1\leq i\leq N(n)}\frac{|\partial T_{in}|}{|T_{in}|}
    \end{equation}
In particular, if $\max\limits_{1\leq i\leq N(n)}\frac{|\partial T_{in}|}{|T_{in}|}$ is asymptotic to $\mathcal{I}_G(n)$, then $\mathcal{I}_{G\act X}(n)$ is asymptotic to $\mathcal{I}_G(n)$.
\end{thm}

Take $G=\Z^d$ for an easy example. Then one of the tilings is the collection of hypercubes. Moreover, the isoperimetric ratio of such tilings is are exactly same as the isoperimetric ratio of this group. So for $G=\Z^d \act X$, we have $\mathcal{I}_{G\act X} \sim \mathcal{I}_G$ and the value will could be easily shown which is asymptotic to $\frac{2d}{n^{1/d}}$. There is one more example following the \ref{cor2} which discuss about the actions of discrete Heisenberg group.

In general, we can also extend the definition of $F_{G\act X}$ from actions to the graphing. But currently we don't know which is the concept for graphings to correspond to being "free" for group actions. Therefore, the generalization of most theorems in this papers are not yet clear.

The paper is organized as follows: in Section \ref{background}, then we will introduce our new isoperimetric profiles in Section \ref{newdefn} and prove Theorem \ref{iff} in Section \ref{proveiff}. In Section 5, we will talk about the case when the action is probability measure preserving, including Theorem \ref{sm} and Theorem \ref{folnertilings}.

\subsection*{Acknowledgement}
Thanks to my supervisor Wouter Van Limbeek for providing me suggestions on formatting my first paper and checking the grammar mistakes. In addition, thanks to Tianyi Zheng and Mikolaj Fraczyk for suggestions about for paper.

\section{Background}\label{background}
\subsection{Isoperimetric profiles for the groups}
Now we first define the isoperimetric profiles of finitely generated groups. Let $S$ be a generating set of the group $G$. We say $S$ is \emph{symmetric} if $ s\in S$ implies $s^{-1}\in S$.
\begin{defn}
Assume $G$ is an finitely generated group and $S$ is a finite symmetric generating set. The isoperimetric profile of $G$ with respect to $S$ is defined as 
\[
\mathcal{I}_{(G,S)}(n)=\inf_{|F|\leq n}\frac{|\partial F|}{|F|}
\]
We write $\mathcal{I}_G(n)$ if there is no confusion.
\end{defn}
\noindent\textbf{Remark.} The isoperimetric profile for the group $G$ is independent of the set $S$ up to a multiplicative constant.\\ 
Let us recall the amenability defined via F\o lner sets and the isoperimetric profile for groups.
\begin{defn}\label{amen}
A discete group $G$ is amenable if for any finite subset $E\subset G$ and $\eps>0$, there exists a finite subset $F\subset G$ such that 
\[
|gF\Delta F|\leq \eps |F| \quad \forall g\in E
\]
\end{defn}
\noindent If $G$ is generated ny a finite symmetric set $S$, then $G$ is amenable if and only if
\[
\mathcal{I}_G(n)=\inf_{|F|\leq n}\frac{|\partial F|}{|F|}\xrightarrow{n\to \infty}0
\]
where $\partial F= \{g\in F| sg\notin F, \text{ some } s\in S \}$.
In general, there are also notions of inner and outer boundary. In this paper, $\partial F$ simply refer to the inner boundary $\partial_{in}F$ and outer boundary is defined as $\partial_{out}F\defeq \partial_{in}(F^c)$, and these notion can be also generalized to a $G$-space $X$, where $F$ can be assumed as a subset of $X$\\

For more informations on the amenable groups, see Juschenko's book \cite{juschenko2022amenability}

\begin{defn}
Let $T,K$ be fintie subsets of $G$ and let $\eps>0$. We say $T$ is $(K ,\eps)$-\emph{invariant} if 
\begin{equation}
    \frac{|KT\Delta T|}{|T|}\leq \eps
\end{equation}
\end{defn}
To generalize the Definition \ref{amen}, say the $K-$(inner) boundary of $T$ to be the set $\{g\in T:kg\notin T, \text{ some }k\in K\}$, and the subset which exculde the $K-$boundary called $K$-interior of $T$.

The notion of $(K,\eps)$-invariance is closely related to F\o lner sequences. It is easy to see that a sequence $(F_n)$ is a F\o lner sequence if and only if for every finite $K\subset G$ and $\eps>0$, $(F_n)$ is eventually $(K,\eps)$-invariant. 

\subsection{Graphings, Equivalence Relations}
In order to prove Theorem \ref{iff}, we need to work with measured equivalence relations. Assume $G$ is a finitely generated group acting on a probability measure space $(X,\mathcal{B},\mu)$. We say it is $non$-$singular$ if the action is measure-class preserving, i.e., for any measurable subset $A$ with $\mu(A)=0$ and any $g\in G$, we have $g_*\mu(A)=\mu(g^{-1}A)=0$. For a nonsingular action, the Radon-Nikodym derivatives $\frac{\d g_*\mu}{\d\mu}$ are well-defined and $L^1$. We say that the action $G\act(X,\mu)$ is $essentially$ $free$ if for any non-trivial $g\in G$, $\mu(\text{Fix}(g))=0$, where $\text{Fix}(g)\defeq\{x\in X\mid gx=x\}$. We see that an essentially free action is free $\mu$-a.e., because the action restricted to the co-null subset $X'=X\setminus(\cup_{g\in G}\text{Fix}(g))$ is free.\\

\begin{defn}
Let $(X,\mathcal{B},\mu)$ be a probability measure space and let $\varphi_i:X_i \to X$ be a finite family of non-singular measurable maps defined on some subsets $X_i\in\mathcal{B}$. We say the triple $(X,\mu,(\varphi_i)_{i\in I})$ is a \emph{graphing}. 
\end{defn}
Here we assume the family of maps $(\varphi_i)_{i\in I}$ is symmetric, which means that each $\varphi_i$ is injective and $(\varphi_i)_{i\in I}$ contains each inverse $\varphi_i^{-1}:\varphi_i(X_i)\to X_i$. Then we define the equivalence relation $\mathscr{R}$ to be the orbit equivalence relation generated by $(\varphi_i)_{i\in I}$. For example, for a nonsingular action of a countable group $G$ on $(X,\mathcal{B},\mu)$, let $S\subset G$ be a finite symmetric generating set. Then this action induces a graphing by choosing $(\varphi_s)_{s\in S}$ where $\varphi_s:X\to X$ is the left-translation by $s$.
Recall an equivalence relation $\mathscr{R}$ on $(X,\mathcal{B},\mu)$  is called \emph{measurable} it is also measurable as a subset of $X\times X$. For example, the orbit equivalence relation of $G$ on $X$ is measurable. And $\mathscr{R}$ is $nonsingular$ if all the $\varphi_i$ for $i\in I$ are nonsingular. Given a triple $(X,\mu,\mathscr{R})$ where $\mathscr{R}$ is a nonsingular measurable equivalence relation, we say $\mathscr{R}$ is $hyperfinite$ if there exists a increasing sequence of finite equivalence relations $\mathscr{R}_n$ such that $\mathscr{R}=\cup_n\mathscr{R}_n$. For an equivalence relation, being hyperfinite is the same as being amenable

\begin{thm}[Connes-Feldman-Weiss \cite{connes_feldman_weiss_1981}]
A measurable equivalence relation is hyperfinite if and only if it is amenable
\end{thm}
Actually, the following theorem tells us that under some conditions, the amenability of the measured equivalence relation is equivalent to amenability in the sense of Zimmer \cite{zimmer}
\begin{thm}[Adams-Elliot-Giordano \cite{amenable}]\label{essfree}
Assume $(X,\mu)$ is a standard measure space and $G$ acts ergodically on $(X,\mu)$. Then the action $G\act(X,\mu)$ is amenable in the sense of Zimmer if and only if the equivalence relation $\mathscr{R}_G$ induced by the action of $G$ on $(X,\mu)$ is amenable and the stability subgroup $G_x=\{g\in G;gx=x\}$ is amenable $\mu-a.e$.  
\end{thm}
In particular if the action is essentially free, then the action is amenable if and only if the orbit equivalence relation is amenable.\\
\subsection{Rokhlin Lemma}\label{sectionrokhlin}
Here we first give the definition of tiles and multi-tiles. Let $G$ be a countable amenable group.
\begin{defn}\label{tiling}
We say a finite subset $T\subset G$ is a \emph{tile} if the right translates of $T$ partition $G$, i.e., there exists a subset $ C\subset G$ such that $G=\sqcup_{c\in C}Tc$. This $C\subset G$ is called a tiling center.
\end{defn}

\begin{defn}
A finite collection of finite subsets $\{F_i\mid 0\leq i \leq m\}$ is called a \emph{multi-tile} if $e\in F_i$ for each $0\leq i\leq m$ and there exists a collection of subsets $\{C_i\mid 0\leq i\leq m\}$ called \emph{center set} such that $\{F_ic\mid 0\leq i\leq m, c\in C_i\}$ partition $G$.
\end{defn}

Many groups have the tiles, for example the cyclic groups and solvable groups. In general, all the \emph{elementary amenable groups} have tilings, but currently we still don't know whether every amenable group has a tile. Ornstein-Weiss \cite{bams/1183545203} generalized the Rokhlin Lemma from free $\Z$-actions to free $G$-actions and from first $n$ integers to an abstract tile: For a probability measure preserving free action, the translates of tiles will covers a large portion of the space. Since the existence of tilings is not well known for amenable groups so far, we will consider the notion of the multi-tiles. Downarowicz-Huczek-Zhang's \cite{tilings} proved the existence of multi-tile for amenable groups.
\begin{thm}[Downarowicz-Huczek-Zhang \cite{tilings}]\label{multitile}
Fix $\eps>0$ and a finite set $K\subset G$. There exists a multi-tile $(T_1,\dots, T_n)$ such that $T_i$ is $(K,\eps)$-invariant for each $1\leq i\leq n$.
\end{thm}

Then we can immediately build F\o lner multi-tilings :

\begin{cor}\label{cor}
    Let $G$ be a finitely generated infinite amenable group. Then there exists a sequence of multi-tile $(T_n)$ where shapes are $(T_n)=(T_{1n},\dots, T_{Nn})$ and $N=N(n)$, such that each shape has size at most $n$ and $\max\limits_{1\leq i\leq N}\frac{|\partial T_{in}|}{|T_{in}|}\xrightarrow{n \to \infty} 0$.
\end{cor}

A multi-tile is actually called an \emph{exact quasitiling} in Downarowicz-Huczek-Zhang \cite{tilings} and the notion of ``quasitiling" in paper is modified from Ornstein-Weiss \cite{bams/1183545203}. The following theorem is the generalized Rokhlin Lemma for multi-tiles.

\begin{thm}[\emph{Rokhlin Lemma for Multi-tiles}]\label{rokhlin}
    Let $G$ be an amenable group and $G\act (X,\mu)$ be a free probability measure preserving action. $T=(T_1,\dots,T_n)$ be a multi-tile. Then for any $\eps>0$, there exists a collection of measurable subsets $A=(A_1,\dots,A_n)$ such that the subsets $\{ tA_i \mid t\in T_i \text{ and } 1\leq i\leq n \}$ are disjoint and $\mu(\cup_{i=1}^n T_iA_i)>1-\eps$.
\end{thm}

Acutally, Conley-Jackson-Kerr-Marks-Seward-Tucker-Drob \cite{generaltilings} gives a similar result, which says that all of the space $X$ can be partitioned by some multi-tile. However for us, the advantage of Theorem \ref{rokhlin} is that it applies to any milti-tile, so it will produce stronger bounds on isoperimetric profiles. The following proof of Theorem \ref{rokhlin} is basically generalized from the proof of the Rokhlin Lemma for tiles in Ornstein-Weiss \cite{bams/1183545203}. We first introduce some concepts from Ornstein-Weiss which help us to finish the proof.\\
    Let $H\subset G$ be a finite set. An $H-set$ is a set $A\subset X$ such that $\{hA: h\in H\}$ are pairwise disjoint. We also call $HA$ an $H$-tower with $A$ being the \emph{base} of the tower.
\begin{proof}[Proof of Theorem  \ref{rokhlin}]
    Suppose there is a free measure preserving action $G\act(X,\mathcal{B},\mu)$ and given $\eps>0$, fix $k>10/\eps$. By the amenability of $G$, there exists a F\o lner set $H$ that is $(T_iT_i^{-1}T_jT_j^{-1},1/(10k^2))$-invariant\footnote{The notion of $(K,\eps)$-invariant in the sense of Ornstein-Weiss is different, but when $K$ contains identity, it coincides with ours up to multiplication by the size of the generating set.}. Then there exists center subsets $C_1,\dots, C_n \subset H$ such that for each $1\leq i\leq n$, $C_i$ is a $T_i$-set, and $T_1C_1,\dots,T_nC_n$ cover the entire $T_iT_i^{-1}T_jT_j^{-1}$-interior of $H$ for each $i.j$.
    (For brevity, we just say "interior" in the rest of the proof).

    According to the lemma 3 \& 4 of Section 2.2 from Ornstein-Weiss \cite{bams/1183545203}, there exists a partition of $X$ into $H$-sets $X=\sqcup_{i=1}^\infty U_i$ and for some number $N$ we have
    \begin{equation}
        \mu(\cup_{i=1}^NHU_i)>1-1/k
    \end{equation}
    \begin{equation}\label{totallost}
        \sum_{i=1}^N \mu(HU_i)<k
    \end{equation}
\end{proof}
Now consider the following $T_j$-sets. For $1\leq j\leq n$, set $W_{j1}=C_jU_1$, and set $W_1=\cup_{j=1}^nT_jW_{j1}$. Note that $W_1$ covers the entire interior of $HU_1$. Now consider $HU_2$ and set 
\begin{equation}
    W'_{j1}=\{x\in C_jU_1\mid T_jx\cap(\cup_{i=1}^nT_iC_i)U_2=\varnothing\}
\end{equation}
for $1\leq j\leq n$. Then let $W_{j2}=W'_{j1}\cup C_jU_2$.
Clearly, $W_{j2}$ is a $T_j$-set and $\sqcup_{j=1}^nT_jW_{j2}$ covers all interior of $HU_2$.\\
Here we show $\sqcup_{j=1}^nT_jW_{j2}$ also covers the interior of $HU_1$. Suppose $y\in(\cup_{j=1}^nT_jC_j)U_1\setminus \cup_{j=1}^n T_jW_{j2}$. Then $y\notin\cup_{j=1}^nT_jC_jU_2$ and $y\notin\cup_{j=1}^nT_jW'_{j1}$. This implies that for any $1\leq j\leq n$, $y\notin T_jW'_{j1}$ and $T^{-1}_jy\notin W'_{j1}$. Hence $T_jT^{-1}_jy\notin T_jC_jU_2$, while $y\notin\cup_{j=1}^nT_jC_jU_2$. So $y$ is in the outer $T_jT^{-1}_j$-boundary of $T_iC_iU_2$, and hence still in the $T_iT_i^{-1}T_jT_j^{-1}$-boundary of $HU_2$. Therefore we proved that $\sqcup_{j=1}^nT_jW_{j2}$ covers the interior of $HU_1$.\\
By induction, for all $1\leq k < N$, we can define 
\begin{equation}
W'_{jk}=\{x\in W_{jk}\mid T_jx\cap (\cup_{i=1}^nT_iC_i)U_k=\varnothing\}
\end{equation}
and 
\begin{equation}
    W_{j,k+1}=W'_{jk}\cup C_jU_{k+1}
\end{equation}
Then just like the case of $k=1$ above, we can also show that $W_{j,k+1}$ is a $T_j$-set and $\cup_{j=1}^nT_jW_{j,k+1}$ covers $\cup_{j=1}^n(T_jC_jU_{k+1})$ and is contained in the $T_iT_i^{-1}T_jT_j^{-1}$-boundary of $HU_k$. It also covers  $\cup_{j=1}^n(T_jC_jU_t)$ for $1\leq t\leq k$ and is contained in boundary of $HU_{t+1}$. This induction ends on $k=N$. Therefore, $\sqcup_{j=1}^nT_jW_{jN}$ covers the entire interior of $HU_i$ for all $1\leq i\leq N$ while losing some points on the boundary of $HU_i$. Therefore, the induction tell us that the total lost is contained in the boundary of each $HU_i$ for $1\leq i\leq N$, which is estimated by (\ref{totallost}):
\begin{equation}
    \frac{1}{10k^2}\sum_{i=1}^N\mu(HU_i)\leq \frac{1}{10k}<\eps
\end{equation}
Thus all the $tW_{jN}$ for $t\in  T_j$ and $1\leq j\leq n$ are pairwise disjoint with
\begin{equation}
    \mu(\sqcup_{i=1}^n T_iW_{iN})\geq 1-\eps.
\end{equation}

This theorem above will helps us to prove that the isoperimetric profile of an action of amenable group is bounded by the isoperimetric ratio of the multitilings (Theorem \ref{cor2}), which is the main part of the Theorem \ref{folnertilings}. The other part of the Theorem \ref{folnertilings} is actually the corollary \ref{cor}.

\section{Subequivalence relations and the isoperimetric profiles of group actions}\label{newdefn}
Now assume the couple $(G,S)$ is a finitely generated group with a finite symmetric generating set $S$. Suppose the group $G$ acts on a probability measure space $(X,\mu)$. 
\begin{defn}\label{folner}
Let $\theta$ be an equivalence relation of $X$. A $subequivalence$ $relation$ $\eta \subset \theta$ is a subset of $\theta$ which is also an equivalence relation on $X$.
\end{defn}
Let $\theta$ be the orbit equivalence relation of the action $G\act X$. We say a subequivalence relation $\eta\subset \theta$ is \emph{connected} if for any pair $(a,b)\in \eta$, there exists $s_1,...,s_k \in S$ such that $b=s_ks_{k-1}\dots s_1a$ and $(a,s_1a),\ (s_1a,s_2s_1a),...,\ (s_k^{-1}b,b)\in\eta$,
Intuitively, a subequivalence relation of an orbit equivalence relation $\theta$ is connected if by applying generators, one can move from any point to any other point while staying inside the equivalence class. In the view of graph, we can also say for each $x\in X$, $[x]_\theta$ can be regarded as a connected subgraph of the Cayley graph of $Gx$.
In this paper, we assume that all the subequivalence relation of the orbit equivalence relation is measurable.\\
\begin{defn}
Let $A$ be a measurable subset of $X$. Then $\theta\cap(A\times A)$ is an equivalence relation on $A$. We define the $restriction$ $of$  $\theta$ $restricted$ $to$ $A$ to be the largest connected subequivalence relation of $\theta\cap (A\times A)\subset A\times A$.
We denote the restriction of $\theta$ of $\theta$ to $A$ by $\theta||_A$.
\end{defn}
In other words, the equivalence classes of $\theta||_A$ are exactly all the connected components of the equivalence classes of $\theta\cap(A\times A)$.\\

\begin{defn}\label{folneraction}
Fix a finite symmetric generating set $S$ of $G$ and probability measure space $(X,\mu)$. Then the \emph{isoperimetric profile of an action} $(G,S)\act X$ is defined as
\[
\mathcal{I}_{(G,S)\act X}(n)\defeq \inf_{\substack{\theta' \subseteq \theta \\|\theta'|\leq n}}
\mu(\partial \theta')=\inf_{\substack{\theta' \subset \theta\\|\theta'|\leq n}}\mu\{x\in X \mid sx \notin [x]_{\theta'} \text{, for some } s\in S\}
\]
where $\theta'$ is a measured subequivalence relation of the orbit equivalence relation $\theta$ with at most $n$ elements in equivalence classes. $|\theta'|$ is the maximal size of the equivalence classes. We will also write $\mathcal{I}_{G\act X}(n)$ if there is no confusion about the generating set.
\end{defn}
In this paper, when talking about the subequivalence relation of the orbit equivalence relation $\theta$, we use notation $\theta_n$ to represent the subequivalence relation with classes size at most $n$

\noindent\textbf{Remark.} Actually in Definition \ref{folneraction}, we need only consider the connected subequivalence relations in the infimum. We will prove this in Section \ref{special}.\\
\textbf{Remark.} The boundary set $\partial\theta_n=\{x\in X\mid sx\notin[x]_{\theta_n}, \text{ for some }s \in S\}$ above is measure in $X$. To see this, we can actually rewrite it as 
\begin{equation}
    \partial\theta_n=\bigcup_{s\in S}\text{Graph(s)}^{-1}(X\times X\setminus\theta_n)
\end{equation}
where $\text{Graph(s)}:X\to X\times X$ is the graph of the function $x \mapsto sx$. So $\partial \theta_n$ is measurable. Moreover, we can naturally define the concept of the boundary in a $G$-space $X$, directly from the boundary concept in $G$. The next proposition tell us that the isoperimetric profile of an action will be nontrivial as long as the action itself is slightly stronger than non-singular. A special case was studied by Sayang and Shalom \cite{ultralimits}, where they called these actions ``bounded quasi-invariant''. Here I generalize this notion as the following.

\begin{defn}
    Consider the nonsingular group action $(G,S)\act X$ where $S$ is the symmetric generating set. We say this action is $p$\emph{-quasi-invariant} if the Radon-Nikodym derivatives $ds_*\mu/d\mu(x)$ are $L^p$ for each $s\in S$. Throughout the paper, we simply say \emph{``$p$-QI"}. When we say \emph{``$(1+\delta)$-QI"}, we always assume $\delta>0$. When we say ``BQI"(bounded QI), we mean $p=\infty$.
\end{defn}

\begin{prop}
    Assume the finitely generated group $G$ acts on standard probability space $(X,\mu)$ and the action is $(1+\delta)$-QI. Then $\mathcal{I}_{(G,S)\act X}(n)>0$.
\end{prop}
\begin{proof}
    For some $n>0$, we have that $\mathcal{I}_{(G,S)\act X}(n)=0$ if and only if there exists a sequence of subequivalence relations $(\theta_i)$ of the orbit equivalence relation such that $\mu(\partial\theta_i)\to 0$ when $|\theta_i|\leq n$. Lemma $\ref{weakct}$ below implies that $\mu(S^n\partial\theta_i)$ is bounded by a polynomial in $\mu(\partial \theta_i)$. Therefore, $\mu(S^n\partial\theta_i) \to 0$ as well. But $|[x]_{\theta_i}|\leq n $ implies $[x]_{\theta_i}\subset S^nx$. Therefore, $\mu(S^n\partial\theta_i)=\mu(X)=1$, which is a contradiction.
\end{proof}

Unlike the usual isoperimetric profile for groups $G$, the isoperimetric profile of an action is not independent of the generating set in general. This is because in the isoperimetric profiles for groups, we use the counting measure and the action of $G$ on itself by translation is measure-preserving. If the action is BQI , then the rate of change is independent of the generating sets up to a multiplicative constant, i.e., if $S_1$ and $S_2$ are symmetric finite generating sets and $\frac{\d s_*\mu}{\d\mu}$ is $L^\infty$-bounded for each $s\in S_1,S_2$, then there exists a constant $C>0$ such that for every $n\geq 1$,
\[
\frac{1}{C}\mathcal{I}_{(G,S_1)\act X}(n)\leq\mathcal{I}_{(G,S_2)\act X}(n)\leq C \mathcal{I}_{(G,S_1)\act X}(n)
\]
\indent In general, if the Radon-Nikodym derivatives are $L^p$, $p>1$, then $\mathcal{I}_{G\act X}$ is disturbing when changing the different symmetric generating set but the extent of such disturbance depends on $p$. We say two real number $p,q\in [1,\infty]$ are \emph{dual} to each other if $1/p+1/q=1$.
\begin{prop}\label{generatingset}
Assume that the finitely generated group $G$ acts on the probability measure space $(X,\mu)$ and the action is essentially free. Let $S_1, S_2$ be any two finite symmetric generating sets. If $(G,S_1\cup S_2)\act X$ is p-QI, then there exists a constant $C$ such that 
\[
\mathcal{I}_{(G,S_2)\act X}(n) \leq C(\mathcal{I}_{(G,S_1)\act X}(n))^{1/q}, \quad \text{ where } q \text{ is dual to }p
\]
In particular, if $(G,S_1\cup S_2)\act X$ is BQI, then there exist constants $C_1, C_2>0$ such that
\[
C_1(\mathcal{I}_{(G,S_2)\act X}(n))\leq \mathcal{I}_{(G,S_1)\act X}(n) \leq C_2(\mathcal{I}_{(G,S_2)\act X}(n))
\]
\end{prop}

In Definition \ref{folneraction}, we regard the boundary of a set as inner boundary. Alternatively, we could also use the outer boundary to define the isoperimetric profile of an action. In this case, all the theorems and properties arises in this paper remains same. Moreover, say $\mathcal{I}_I(n)$ and $\mathcal{I}_O(n)$ to be the isoperimetric profile of same action defined according to inner and outer boundary repectively. Then using proposition \ref{generatingset}, it is easy to show that if the action is BQI then they have the same rate of decay up to a multiplicative constant; if the action is p-QI, then they have the same rate of decay up to $q^{th}$ power, where $q$ is dual to $p$.  


\section{Proof of main theorems}\label{proveiff}

\indent Before proving Theorem \ref{iff}, we first introduce a useful lemma by Fraczyk and Van Limbeek, which says that an amenable equivalence relation is always finite on an arbitrarily large subset.

\begin{lem}\label{lem3.7}(Fraczyk-Van Limbeek \cite{https://doi.org/10.48550/arxiv.1905.13584})
Let $(X,\mu, (\varphi_i)_{i\in I})$ be a finite symmetric graphing generating a nonsingular amenable measured equivalence relation $\mathcal{R}$. Then for every $\eps>0$ there exists $M \in \N$ and a subset $Z\subset X$ such that $\mu(Z)\geq 1-\eps$ and the equivalence relation $\mathcal{E}$ on $Z$ generated by the restrictions of $(\varphi_i)_{i\in I}$ satisfies $|[z]_{\mathcal{E}}|\leq M$ for $\mu$-a.e. $z\in Z$.
\end{lem}

\begin{lem}\label{weakct}
Let $G$ be a finitely generated group acting on the probability measure space $(X,\mu)$. Let $S$ be a finite symmetric generating set of $G$. Assume that $(G,S_1\cup S_2)\act X$ is p-QI, where $p=1+\delta,$ where $ \delta>0$. Then for any sequence of measurable subsets $\{A_n\}_{n\in\N}$ with $\mu(A_n)\rightarrow 0$, we have $\mu(sA_n)\rightarrow 0$ for any $s\in S$.
\end{lem}

\begin{proof}
Suppose that $\frac{\d s_* \mu}{\d \mu}$ is $L^{1+\delta}$, so we have $\mu(sA_n)=\int_{A_n}\frac{\d s_*\mu}{\mu}\d \mu$. Then by H$\ddot{\text{o}}$lder's inequality, for $1\leq q <\infty$ and $p$ such that $1/p+1/q=1$, we have
\begin{align}
    \mu(sA_n)=\int_{A_n}\frac{\d s_*\mu}{\d\mu}\d \mu\leq &\norm{\frac{\d s_*\mu}{\d\mu}}_p\norm{1_{A_n}}_q\label{holder1}\\
    =&\norm{\frac{\d s_*\mu}{\d\mu}}_p(\mu(A_n))^{1/q},\quad \label{holder2}
\end{align}
Using $p=1+\delta$ shows $\mu(sA_n)\to 0$. 
\end{proof}

Now we are ready to prove our main theorems.
\begin{proof}[Proof of Theorem \ref{iff} ($\Rightarrow$)]
Consider a non-singular amenable $G$-space $(X,\mu)$ with finite symmetric generating set $S$. This induces a graphing $(X, \mu, (\varphi_s)_{s\in S})$ where the maps $(\varphi_s)_{s\in S}$ are the translations of $X$ by $s\in S$. By Theorem \ref{essfree}, we know that the equivalence relation generated by the graphing is amenable. Let $n\geq 1$. Lemma \ref{lem3.7} tells us that for any $\eps>0$, there exists a subset $Z\subset X$ such that $\mu(Z)>1-\eps$ and the restriction of the orbit equivalence classes $\theta$ on $Z$ has size at most $n$, i.e.,$\left|[x]_{\theta|_Z}\right|\leq n$.  Now define the new equivalence relation $\theta_n\subset \theta$ as follows,
\begin{equation}
     [x]_{\theta_n}=\begin{cases}
    & [x]_{\theta||_{Z}},\quad \text{ if } x \in Z.\\
    & \{x\}, \quad \text{ if } x \notin Z.
    \end{cases}
\end{equation}
Recall that $\theta||_Z$ is the connected subequivalence relation of $\theta$ restricted on $Z$. For any distinct two equivalence classes $[x]_{\theta_n},[y]_{\theta_n}\in X/\theta_n$ where $x,y\in Z$, $\partial_{out}[x]_{\theta_n}$ and $[y]_{\theta_n}$ are disjoint. Therefore, we have that 
\begin{equation}
    \bigcup_{[x]_{\theta_n}\subset Z}\partial[x]_{\theta_n}=\partial Z.
\end{equation}
Now we have the following computation:
\begin{align}
    \mu(\partial\theta_n)=\mu(\bigcup_{[x]_{\theta_n}\in X/\theta_n}\partial [x]_{\theta_n})=&\mu(\bigcup_{[x]_{\theta_n}\subset Z}\partial [x]_{\theta_n}\cup\bigcup_{[x]_{\theta_n}\subset Z^c}\partial [x]_{\theta_n})\\
    =&\mu(\partial Z)+\mu(Z^c)\\
    =&\mu(\partial_{out}Z^c)+\mu(Z^c)\\
    \leq &\mu(SZ^c)+\mu(Z^c)\\
    \leq &|S|\mu(\partial Z^c)^{1/q}+\mu(Z^c)\label{bound}\\
    \leq |S|\eps^{1/q}+\eps
\end{align}
where (\ref{bound}) comes from (\ref{holder2}). In the proof of Lemma \ref{weakct}. Lemma $\ref{lem3.7}$ also tells us that $n\to \infty \implies \eps\to 0$, which finish the last step of the proof of one direction.
\end{proof}

\begin{proof}[Proof of Theorem \ref{iff} ($\Leftarrow$)]
Let $\theta$ be the orbit equivalence relation. Choose an enumeration of the elements of the group $G=\{g_0,g_1,g_2,...\}$, which is increasing with respect to the word metric. Since $G\act X$ is essentially free, such ordering is fixed for a.e. $x\in X$.  In particular, $\norm{g_t}\leq t$. For $n\in\N$ and the finite subequivalence relation $\theta_n$ of the orbit equivalence relation, there exists a sequence $\{k_n\}_{n\in\N}$ such that $[x]_{\theta_n}\subset B(x,n)=\{g_ix\}_{i=0}^{k_n}$ for any $x \in X$. Now for $k<n$ define the k$^{th}$ boundary of $\theta_n$:
\[
\partial^k\theta_n\defeq\{ x\in X \mid \exists g \in G, \norm{g}\leq k \text{ with } gx\notin [x]_{\theta_n} 
\}
\]
By Lemma \ref{weakct}, we have that
\begin{equation}\label{kthbdry}
\mu(\partial^k\theta_n)\leq \mu(\partial\theta_n)+\sum_{s\in S}\mu(s\partial\theta_n)+\dots +\sum_{s_1,\dots, s_k\in S}\mu(s_1...s_k\partial\theta_n)\xrightarrow{n\to \infty} 0
\end{equation}
For each $n\in \N$, define $\alpha^\theta_n:X\to \N$ by $\alpha^\theta_n(x)=\min\{t\mid g_tx\notin[x]_{\theta_n}\}$. Then $\alpha^\theta_n(x)\leq k$ implies that $\norm{g_{\alpha^\theta_n(x)}}\leq k$, hence $x\in \partial^k\theta_n$. Combined with (\ref{kthbdry}), for any $\eps,k>0$, for every large enough $n$ we have $\mu\{x\mid \alpha^\theta_n(x)\leq k\}\leq \mu(\partial^{|S|k}\sigma_n)<\eps$, i.e., $\alpha^\theta_n \xrightarrow{n\to \infty} \infty$ in measure.\\

Moreover, there exists $m>n$ such that $\mu\{x\mid \alpha^\theta_m(x)\leq k_n\}\leq \eps$. For any $g_ix\in[x]_{\theta_n}$, we have $i\leq k_n$. Therefore $\alpha^\theta_m(x)\geq k_n $ implies $ [x]_{\theta_n}\subsetneq [x]_{\theta_m}$. Therefore $\mu\{x\mid [x]_{\theta_n}\not\subseteq [x]_{\theta_m}\}<\eps$.\\
Now by induction and say $\eps_0=\eps$, we have a sequence of equivalence relations $\{\sigma_i\}_{i\in\N}$ such that if $E_i=\{x\mid [x]_{\sigma_i}\not\subseteq [x]_{\sigma_{i+1}}\}$, then $\mu(E_i)<\eps_i$. Choose $\eps_{i+1}=\eps_i/2$ for $i\geq 0$, then for any $x$ in the set $F=\cup_i(X\setminus E_i)$ with measure $\mu(F)>1-\eps$, we have a strictly increasing sequence of equivalence classes $[x]_{\sigma_1}\subsetneq [x]_{\sigma_2}\subsetneq[x]_{\sigma_3}\subsetneq\dots$.\\
We say $F_n=\cup_{i=n}^\infty(X\setminus E_i)$ and define a new sequence of equivalence relations $\tau_n$:
\begin{equation}
    [x]_{\tau_n}=\begin{cases}
    [x]_{\sigma_n|_{F_n}}, & \text{ if $x\in F_n$ }\\
    \{x\}, & \text{ if $x\notin F_n$}
    \end{cases}
\end{equation}
We see that $\{F_n\}_{n\in\N}$ is increasing: For any $n\in\N$, if $x\notin F_n$ then $[x]_{\tau_n}=\{x\}\subset[x]_{\tau_{n+1}}$; if $x\in F_n$, then $[x]_{\sigma_n|_{F_n}}\subseteq[x]_{\sigma_{n+1}|_{F_{n+1}}}$. Hence $\tau_n$ is increasing. For $\sigma_n$, we have
$\alpha^\sigma_n\to \infty $ in measure. 
Since $\{\sigma_n\}_{n\in\N}$ is a subsequence of $\{\theta_n\}_{n\in\N}$, there exists a subsequence $\{s_n\}_{n\in\N}\subseteq\{k_n\}_{n\in\N}$ such that 
\begin{equation}
\{x\mid\alpha_n^\sigma(x)\geq s_n\}\subseteq\{x\mid[x]_{\sigma_n}\subseteq[x]_{\sigma_{n+1}}\}\subseteq F_n
\end{equation} 
Hence for any $g\in G$, there exists $n\in\N$ such that $gx\in F_n$ and $gx\in[x]_{\sigma_n|_{F_n}}$. Therefore since $\mu(F_n)\to 1$, we have that $\cup_n\tau_n$ is exactly the orbit equivalence relation, and hence a hyperfinite equivalence relation.
\end{proof}

Assume $m$ is the measure supposed in the Remark following Theorem \ref{iff}. Then for any measurable $A\subset X$, we have
\begin{align}
    \mu(A)=\sum_{s\in S}\mu(s^{-1}A)m(s)\geq \mu(s^{-1}A)m(s)\\
    \implies m(s)\leq \frac{\mu(sA)}{\mu(A)}\leq \frac{1}{m(s^{-1})}\\
    \implies m(s)\leq \frac{\d s_*\mu}{\d \mu}\leq \frac{1}{m(s^{-1})}
\end{align}

The following is the proof of Proposition \ref{generatingset}, which is actually similar to the proof of lemma \ref{weakct}.
\begin{proof}
Assume $\frac{\d s_*\mu}{\d\mu}$ is $L^\infty$-bounded for each $s\in S_1,S_2$. Then there exists $ k\in \N$ such that $S_1^k\supseteq S_2$, hence we have
\begin{equation}
    \partial_{S_2}\theta_n\subset\partial_{S^k_1}\theta_n\subset S^{k-1}_1\partial_{S_1}\theta_n
\end{equation}
where the $\partial_{S_1}$ and $\partial_{S_2}$ represent the boundary of a set with respect to the different generating sets. Set $M=\max\limits_{s\in S_1}{\|\frac{\d s_*\mu}{\d\mu}\|_{\infty}}$, then $\mu(\partial_{S_2}\theta_n)\leq M^{k-1}\mu(\partial_{S_1} \theta_n)$
If $\frac{\d s_*\mu}{\d\mu}$ is $L^p$-bounded for some $p>1$,set $M_p=\max\limits_{s\in S_1}{\|\frac{\d s_*\mu}{\d\mu}\|_{p}}$. From H$\ddot{o}$lder's result (\ref{holder1}) and (\ref{holder2}) we have
\begin{equation}
    \mu(\partial_{S_2}\theta_n)\leq\mu(S_1^{k-1}\{\partial_{S_1}\theta_n\})\leq  M_p^{k-1}(\mu(\partial_{S_1}\theta_n))^{1/q}
\end{equation}
where $q$ is dual to $p$.
\end{proof}
\section{Some Special Cases of Actions}\label{special}
In this section, we will obtain a stronger result if the action is measure-preserving. We will also compare the isoperimetric profiles of group actions with the classical isoperimetric profiles of groups.
\begin{prop}\label{folnerrelation}
Assume the action of the finitely generated group $G$ on $ (X,\mu)$ is measure-preserving and essentially free. Then $\mathcal{I}_{G\act X}(n)\geq \mathcal{I}(n)$ for $n$ large enough.
\end{prop}

\begin{proof}
Let $\theta_n\subset \theta$ be an subequivalence relation of the orbit equivalence relation such that $\left|\theta_n\right|\leq n$. For each equivalence class $[x]\defeq[x]_{\theta_n}$, define the measure $\mu_{[x]}$ to be the counting measure on $[x]$. Consider the projection $\pi:X\to X/\theta_n$, then by the disintegration we have
\[
0<\mu(\partial \theta_n)= \int_{X/\theta_n}\mu_{[x]}(\partial \theta_n)\d\pi_*\mu([x])=\int_{X/\theta_n}\frac{|\partial [x]|}{|[x]|}\d\pi_*\mu([x])= \int_X\frac{|\partial[x]|}{|[x]|}\d \mu(x)
\]
hence there exists $x\in X$ such that $|\partial[x]|\leq \mu(\partial\theta_n)|[x]|$. Since the action is free, $g\mapsto gx$ is bijective for a.e. $x\in X$. So we can choose subset $F\subset G$ such that $Fx=[x]$, then $\mu(\partial\theta_n)\geq\frac{|\partial F|}{|F|}$ and hence $\mathcal{I}_{G\act X}(n)\geq \mathcal{I}(n)$
\end{proof}

By Proposition \ref{folnerrelation}, $\mathcal{I}_{G\act X}(n)\to 0 $ implies $ \mathcal{I}_G(n)\to 0$. On the other hand, if $G$ is amenable, then any group action is amenable and the stabilizer groups are amenable. Therefore, if the action is ergodic then Theorems \ref{iff} and \ref{essfree} tells us the decay of $\mathcal{I}_{G\act X}(n)$. So we have the following
\begin{cor}
Assume $G$ is finitely generated and $G\act(X,\mu)$ is a probability measure preserving and essentially free action. If $\mathcal{I}_{G\act X}(n)\to 0$, then $G$ is amenable. The converse holds if the action is also ergodic.
\end{cor}

The next lemma tells us that we need only consider connected subequivalence relations to compute the isoperimetric action function. Say $\Theta$ is the set of subequivalence relations of the orbit equivalence relation $\theta$ and $\Theta_c\subset \Theta$ are the connected ones. Then each disconnected $\varphi\in\Theta$ corresponds to a connected equivalence relation with smaller equivalence class size, and measure of the boundary will not decrease. Therefore we have

\begin{lem}
$\inf\limits_{\theta_n\subset \Theta}\mu(\partial \theta_n)=\inf\limits_{\theta_n\subset \Theta_c}\mu(\partial \theta_n)$
\end{lem}

Now, according to what we know from Section \ref{sectionrokhlin}, we can give an upper bound for the isoperimetric profile of actions which is related to the multi-tiles and the isoperimetric profiles of groups.
\begin{thm}\label{cor2}
Let G be a finitely generated amenable group and $(T_n)=(T_{1n},\dots,T_{Nn})$ a sequence of multi-tiles with $|T_{in}|\leq n$ for $1\leq i\leq N$. Then for any free probability measure preserving action $G\act (X,\mu)$, we have 
    \begin{equation}
        \mathcal{I}(n)\leq \max_{1\leq i\leq N}\frac{|\partial T_{in}|}{|T_{in}|}
    \end{equation}
\end{thm}

Now we can prove Theorem \ref{cor2}.
\begin{proof}
Fix $n$ and consider the multi-tile $(T_1,\dots,T_N)$ with shapes $|T_i|<n$. By Theorem \ref{rokhlin}, for any $\eps>0$ there exist $A_1,\dots,A_N$ such that $\{t_iA_i\mid t_i\in T_i, 1\leq i\leq N \}$ are pairwise disjoint and $\mu(\cup_{i=1}^NT_iA_i)>1-\eps$. Now define the equivalence relation $\varphi$ as follows
\begin{equation}
    [x]_{\varphi}=
    \begin{cases}
        T_ia_i, &\text{ if } x\in t_iA_i \text{ for some } a_i\in A_i, 1\leq i\leq N.\\
        \{x\}, &\text{ otherwise. }
    \end{cases}
\end{equation}
Now the equivalence classes of $\varphi$ have size at most $n$. Therefore,
\begin{align}
    \mu(\partial \varphi_n)= & \int_X\frac{|\partial[x]|}{|[x]|}\d\mu=\sum_{i=1}^N\int_{T_iA_i}\frac{|\partial[x]|}{|[x]|}\d\mu+\int_{X\setminus\sqcup T_iA_i}\frac{|\partial[x]|}{|[x]|}\d\mu\\
    \leq & \sum_{i=1}^N\int_{T_iA_i}\frac{|\partial{T_i}|}{|T_i|}\d\mu +\eps\\
    = & \sum_{i=1}^N\frac{|\partial T_i|}{|T_i|}\mu(T_{i}A_i)+\eps\\
    \leq & \max\limits_{1\leq i\leq N}\frac{|\partial T_i|}{|T_i|}(1-\eps)+\eps\xrightarrow{\eps\to 0}\max\limits_{1\leq i\leq N}\frac{|\partial T_i|}{|T_i|}
\end{align}
\end{proof}
Hence Theorem \ref{folnertilings} is a corollary of Theorem \ref{cor2} and Theorem \ref{rokhlin}. In particular, let's consider a probability measure preserving free group action. If a sequence of F\o lner multitilings has F\o lner rates which are asymptotic to the isoperimetric profile of the group, then the isoperimetric profile of the group will be asymptotic to the isoperimetric profile of the action.\\

\noindent \textbf{Remark.} There are many groups that have a sequence of tiles with the isoperimetric ratio asymptotic to $\mathcal{I}_G(n)$. A trivial example is $\Z$ or free abelian groups. Although it is not clear which groups have this property, some non-abelian nilpotent groups also satisfy it.\\
\begin{ex}
    An example for non-abelian nilpotent group.
\end{ex}
Consider the discrete Heisenberg group $H_3(\Z)=\{(\begin{smallmatrix}
1 & a & c\\
0 & 1 & b\\
0 & 0 & 1
\end{smallmatrix}) =y^bz^cx^a\mid a,b,c\in\Z\}$. Consider the bijection $H_3(\Z)\to \Z^3$ with $y^bz^cx^a\mapsto (a,b,c)$. We know that if $b=0$ then $(\{(a,0,c)\}, \cdot)\cong (\Z^2,+)$, and for the other direction $(a,b,c)(0,\pm1,0)=(a,b\pm1,a\pm c)$. We choose $F_n=[0,n]\times[0,n]\times[0,n^2]\cap\Z^3$. Translating this cuboid $F_n$ by $(k(n+1),0,0)$ or $(0,0,k(n+1)), k\in\Z$ will still result in the same shapes, just like in $(\Z^3, +)$. If we translate by $(0,k(n+1),0)$, then consider $F_n=\sqcup_{i=0}^n{F_{in}}$, where $ F_{in}=\{(a,b,c)\in F_n \mid a=i\}$, and we have
\begin{equation}
(0,k(n+1),0)F_{in}=\{(a,b+k(n+1),c+k(n+1)i)\in\Z^3 \mid (a,b,c)\in F_{in} \}
\end{equation}
So clearly, for each $n\in\N$, then $F_n$ is a tile with the center set
\begin{equation}
C_n={a_n\cdot b_n \cdot c_n:a_n\in A_n, b_n\in B_n,c_n\in C_n},
\end{equation}
where 
\begin{align*}
    &A_n=\{(0,k_1(n+1),0)k_1\in \Z\}\\
    &B_n=\{(k_2(n+1),0,0)k_2\in \Z\}\\
    &A_n=\{0,0,(k_1(n+1))k_3\in \Z\}
\end{align*}
By computation, it is easy to find that $\frac{|\partial F_n|}{|F_n|}=\frac{4n^2+2n+5}{(n+1)(n^2+1)}\sim \frac{4}{n}$. According to the Theorem \ref{iff} from Stankov's paper \cite{BSgroupfolner}, which gives a lower bound of the $\mathcal{I}_G(n)$, this convergence rate will be exactly the same as $\mathcal{I}(n)$ up to a multiplicative constant.

\bibliography{refs}
\nocite{*}
\end{document}